\documentclass{article}

\usepackage[margin=1in,footskip=0.25in]{geometry} \usepackage[dvipsnames,table]{xcolor}             \DeclareMathAlphabet{\mathcal}{OMS}{cmsy}{m}{n}

\usepackage[sort&compress,square,comma,numbers]{natbib} \usepackage[pagebackref]{hyperref}                \usepackage{url}                                  \usepackage{breakcites}                           

\usepackage[nosumlimits]{amsmath} \usepackage{amsthm}               \usepackage{amssymb}              \usepackage{amsfonts}             \usepackage{bm}                   \usepackage{bbm}                  \usepackage{mathtools}            \usepackage{nicefrac}             \usepackage{stmaryrd}             \usepackage{commath}              

\usepackage{graphicx}       \usepackage{float}          \usepackage{caption}        \usepackage{subcaption}     \usepackage{booktabs}       \usepackage[most]{tcolorbox}  

\usepackage{algorithm}      \usepackage{algpseudocode}  \usepackage{algorithmicx}   

\usepackage{enumitem}       \usepackage{xspace}         \usepackage{relsize}        \usepackage{titlesec}       \usepackage[normalem]{ulem} \usepackage{contour}        \usepackage{cancel}         \usepackage{todonotes}      \usepackage[capitalize,nameinlink,noabbrev]{cleveref} \usepackage{crossreftools}

\hypersetup{
    colorlinks=true,        linkcolor=PineGreen,    citecolor=NavyBlue,     filecolor=Purple,       urlcolor=Purple         }

\renewcommand*{\backref}[1]{}
\renewcommand*{\backrefalt}[4]{\ifcase #1 (No citations.)\or
        (Cited on page~#2.)\else
        (Cited on pages~#2.)\fi
}

\titlespacing*{\section}{0pt}{12pt}{5pt}
\titlespacing*{\subsection}{0pt}{11pt}{5pt}
\titlespacing*{\subsubsection}{0pt}{11pt}{5pt}
\titlespacing*{\paragraph}{0pt}{6pt}{1em}
\floatstyle{plaintop}
\restylefloat{table}

\algrenewcommand\alglinenumber[1]{\sf\scriptsize\color{NavyBlue}{#1}} 

\let\oldthebibliography\thebibliography
\renewcommand{\thebibliography}[1]{\oldthebibliography{#1}\setlength{\itemsep}{0pt}\setlength{\parskip}{0pt}}

\newtheorem{theorem}{Theorem}[section]

\newtheorem{lemma}[theorem]{Lemma}
\newtheorem{fact}[theorem]{Fact}

\newtheorem{importedtheorem}[theorem]{Imported Theorem}

\theoremstyle{definition}
\newtheorem{definition}[theorem]{Definition}

\theoremstyle{remark}

\crefname{setting}{setting}{settings}
\Crefname{setting}{Setting}{Settings}
\crefname{problem}{problem}{problems}
\Crefname{problem}{Problem}{Problems}
\Crefname{equation}{Equation}{Equations}
\Crefname{importedtheorem}{Imported Theorem}{Imported Theorems}
\Crefname{importedlemma}{Imported Lemma}{Imported Lemmas}

\makeatletter
\newtheorem*{rep@theorem}{\rep@title}
\newcommand{\newreptheorem}[2]{\newenvironment{rep#1}[1]{\def\rep@title{\Cref{##1}, restated}\begin{rep@theorem}}{\end{rep@theorem}}}
\makeatother

\newreptheorem{theorem}{Theorem}
\newreptheorem{lemma}{Lemma}
\newreptheorem{definition}{Definition}
\newreptheorem{corollary}{Corollary}

\contourlength{0.8pt}

\makeatletter
\def\hlinewd#1{\noalign{\ifnum0=`}\fi\hrule \@height #1 \futurelet
	\reserved@a\@xhline}
\makeatother

\DeclareMathOperator*{\E}{\mathbb{E}}
\DeclareMathOperator{\tr}{tr}

\DeclareMathOperator{\TV}{TV}
\newcommand{\KL}{\mathrm{D}_{\mathrm{KL}}}

\newcommand{\defeq}{\ensuremath{\;{\vcentcolon=}\;}\xspace}
  \newcommand{\eps}[0]{\varepsilon} \let\epsilon\eps

\let\norm\relax
\newcommand{\norm}[1]{\enVert{#1}}

\newcommand{\mat}[1]{\bm{#1}}

\newcommand{\etal}{et al.\xspace}

\newcommand{\mA}{\ensuremath{\mat{A}}\xspace}

\newcommand{\mG}{\ensuremath{\mat{G}}\xspace}
\newcommand{\mH}{\ensuremath{\mat{H}}\xspace}
\newcommand{\mI}{\ensuremath{\mathbf{I}}\xspace}

\newcommand{\mW}{\ensuremath{\mat{W}}\xspace}

\newcommand{\cA}{\ensuremath{{\mathcal A}}\xspace}

\newcommand{\cN}{\ensuremath{{\mathcal N}}\xspace}
\newcommand{\cO}{\ensuremath{{\mathcal O}}\xspace}
\newcommand{\cP}{\ensuremath{{\mathcal P}}\xspace}
\newcommand{\cQ}{\ensuremath{{\mathcal Q}}\xspace}

\newcommand{\cS}{\ensuremath{{\mathcal S}}\xspace}

\newcommand{\bbR}{\ensuremath{{\mathbb R}}\xspace}

\newcommand{\rC}{\ensuremath{\mathrm{C}}\xspace}

\title{A non-asymptotic bound on the TV distance between a \\Wishart matrix and an appropriately scaled GOE matrix.}

\author{Raphael A. Meyer\thanks{Department of Statistics, University of California Berkeley and International Computer Science Institute, Berkeley, CA 94720 USA 
(\href{mailto:ram900@berkeley.edu}{ram900@berkeley.edu})
}}

\date{\today}

\begin{document}

\maketitle

\begin{abstract}
    In this note, we prove a non-asymptotic version of a theorem by R\'acz and Richey, showing that a Wishart matrix is close in total variation to an affine transformation of a GOE matrix.
    The proof mirrors a proof in a paper by Bubeck, Ding, Eldan, and R\'acz, with some changes made to make it non-asymptotic.
\end{abstract}

\section{Introduction}
In this note, we prove a non-asymptotic version of \cite[Thm.~1.2]{racz2019smooth} by R\'acz and Richey.
A rich line of research, including but not limited to \cite{bubeck2016testing,bubeck2018entropic,chetelat2019middle,racz2019smooth,brennan2021finetti,bourguin2021limiting,mikulincer2022clt,nourdin2022asymptotic}, has shown that results like ours hold in the asymptotic regime (i.e. in the limit as \(n \to \infty\)).
However, we are not aware of any existing works that prove this result non-asymptotically.
The proof mirrors the proof of Theorem 4.1 in a paper by Bubeck, Ding, Eldan, and R\'acz \cite{bubeck2016testing}, with some changes made to make it non-asymptotic.

This non-asymptotic result is used in a recent paper of Derezi{\'n}ski, Epperly, and Meyer \cite{derezinski2026matrix} as well as a paper of Musco, Musco, Shah, Urschel, and West \cite{musco2026spectral}.
To state the result, we will have to define two random matrices:
\begin{definition}[GOE Matrix]
	A random matrix \(\mG\in\bbR^{n \times n}\) is drawn from the \emph{Gaussian orthogonal ensemble} if it is a symmetric matrix with independent entries on the upper triangle such that for any \(i < j\)
	\[
		[\mG]_{ii} \sim \cN(0,2)
		\qquad
		\text{and}
		\qquad
		[\mG]_{ij} \sim \cN(0,1).
	\]
	We will denote this random matrix as \(\mG \sim \mathrm{GOE}(n)\).
	A GOE matrix is also distributed as \(\mG = \frac1{\sqrt 2} (\mH + \mH^\top)\) where \(\mH\in\bbR^{n \times n}\) has iid \(\cN(0,1)\) entries.
\end{definition}
\begin{definition}[Wishart Matrix]
	Let \(\mH\in\bbR^{n \times k}\) be a matrix with iid \(\cN(0,1)\) entries.
	Then, the matrix \(\mW = \mH\mH^\top \in \bbR^{n \times n}\) is a \emph{Wishart matrix}.
	We denote this random matrix as \(\mW \sim \mathrm{Wishart}(n,k)\).
\end{definition}
We will prove the following statement:
\begin{theorem}
	\label{thm:main}
	Let \(\mG \sim \mathrm{GOE}(n)\) and \(\mW \sim \mathrm{Wishart}(n,k)\).
	Then, if \(k \gtrsim n^3\), the total variation distance between \(\mW\) and \(\widehat\mG \defeq \sqrt k \mG + k\mI\) is at most
	\[
		\mathrm{TV}(\mW, \widehat\mG) \lesssim \sqrt{\frac{n^3}{k}}.
	\]
\end{theorem}
The original theorem of Bubeck \etal \cite{bubeck2016testing} says that \(\mathrm{TV}(\mW, \widehat\mG) \to 0\) if \(k^3/n \to \infty\).
A trivial adaptation of the proof technique used in their paper would show the TV distance is bounded by \(\sqrt{n^3k^{-1} \log(n^3k^{-1})}\).
We slightly alter their proof, using the KL Divergence and Pinsker's Inequality, to avoid this issue.

\section{Preliminaries}

We write \(a \lesssim b\) if \(a \leq \rC b\) for some universal constant \(\rC > 0\), and define \(a \gtrsim b\) analogously.
We write \(a \asymp b\) if \(a \lesssim b \lesssim a\).

\subsection{Information Theory}

We begin by defining the two fundamental information-theoretic metrics.
Recall that for any two distributions \(\cP\) and \(\cQ\), a coupling is a joint probability distribution \((x,y)\) such that the marginal distributions satisfy \(x\sim\cP\) and \(y\sim\cQ\).
The total variation distance has many equivalent definitions.
One says that distributions are close in total variation iff there exists a coupling that makes \(x = y\) with high probability.

\begin{definition}[Total Variation]
    Let \(X\) and \(Y\) be random variables over a shared space.
    Then the \emph{total variation} distance between \(X\) and \(Y\) is
    \[
        \TV(X,Y) = \inf_{\text{Couplings} (X,Y)} \Pr[X \neq Y].
    \]
\end{definition}

It will be helpful to condition our total variations based off of high probability events.
For this, we will need to know the TV Triangle inequality
\begin{align}
    \TV(X, Y) \leq \TV(X, Z) + \TV(Z, Y),
    \label{tv:triangle-ineq}
\end{align}
which holds for any random variables \(X\), \(Y\), and \(Z\).
We will also need the KL Divergence, which is often much easier to analyze than the TV distance.

\begin{definition}[KL Divergence]
    Let \(X\) and \(Y\) be random variables with pdfs \(f_{\rm X}\) and \(f_{\rm Y}\) over a shared space.
    Then the \emph{KL Divergence} between \(X\) and \(Y\) is
    \[
        \KL(X \| Y) = \E_{t \sim X}\left[ \ln\left(\frac{f_{\rm X}(t)}{f_{\rm Y}(t)}\right) \right].
    \]
\end{definition}

We will crucially rely on Pinsker's Inequality, which states that
\begin{align}
    \TV(X,Y) \leq \sqrt{\frac12 \KL(X \| Y)}
    \label{pinskers-ineq}
\end{align}
for all random variables \(X\) and \(Y\).
Additionally, it will be helpful to show that conditioning \(X\) on an event that holds with high probability with not severely change the TV distance or the expected value of our random variables.
We start with a claim about expectations on non-negative random variables.

\begin{lemma}[Conditioning cannot explode expectations]
    \label{lem:conditioning-expectation}
    Let \(X \in \bbR\) be a non-negative random variable.
    Let \cA be an event on \(X\) that holds with probability at least \(\nicefrac12\).
    Then,
    \[
        \E[X \mid \cA]
        \leq 2\E[X]
    \]
\end{lemma}
\begin{proof}
    Let \(p \leq \delta\) be the exact probability that event \cA does not occur.
    Then,
    \[
        (1-p)\E[X\mid\cA] + p\E[X\mid\cA]
        = \E[X]
        = (1-p)\E[X] + p\E[X].
    \]
    Dividing by \(1-p\) and rearranging, we get
    \begin{align*}
        \E[X \mid\cA]
        &= \E[X] + \frac{p}{1-p}(\E[X] - \E[X \mid \cA]) \\
        &\leq \E[X] + \frac{p}{1-p}\E[X] \\
        &\leq 2\E[X],
    \end{align*}
    where the first inequality applies the bound \(\delta\) and uses that \(X\) is non-negative.
    The second inequality uses that \(\delta \leq 0.5\).
\end{proof}

We conclude with a statement about total variation on high-probability events.
\begin{lemma}[Total variation respects conditioning]
    \label{lem:conditioning-tv}
    Let \(X\) and \(Y\) be random variables over a shared space.
    Let \(\cA\) denote an event that \(X\) satisfies with probability at least \(1-\delta\).
    Let \(X_+\) be distributed as \(X\) conditioned on \cA.
    Then,
    \[
        \TV(X, Y) \leq \TV(X_+, Y) + \delta.
    \]
\end{lemma}
\begin{proof}
    Begin by constructing a coupling of \(X\) and \(X_+\).
    We construct a pair \((x,x_+)\) by sampling \(x \sim X\).
    If the event \cA holds for \(x\), then take \(x_+ = x\).
    Otherwise, draw \(x_+ \sim X_+\).
    Then return the pair \((x,x_+)\) as our coupling.
    We have
    \[
        \TV(X,X_+)
        \leq \Pr[x \neq x_+]
        \leq 1-\Pr[\cA]
        \leq \delta.
    \]
    Then, by the TV triangle inequality \cref{tv:triangle-ineq},
    \[
        \TV(X, Y)
\leq \TV(X_+, Y) + \delta.
    \]
\end{proof}

\subsection{Random Matrices}

We will need to understand some facts about GOE and Wishart matrices.
\begin{fact}[Wishart pdf]
    Let \(\mW \sim \mathrm{Wishart}(n,k)\).
    Then the pdf of \mW with respect to the Lebesgue measure on symmetric matrices in \(\bbR^{n \times n}\) is
	\[
		f_{\rm wishart}(\mW) = \frac{(\det(\mW))^{\frac12(k-n-1)} \exp(-\frac12 \tr(\mW)) \cdot \mathbbm1_{[\mW \succeq \mat0]}}{2^{\frac12 kn} \pi^{\frac14 n(n-1)} \prod_{i=1}^n \Gamma(\frac{k+1-i}2)},
	\]
	where \(\Gamma(\cdot)\) is the Gamma function.
\end{fact}

\begin{fact}[GOE pdf]
    Let \(\mG\sim\mathrm{GOE}(n)\).
    Then the pdf of \mG with respect to the Lebesgue measure on symmetric matrices in \(\bbR^{n \times n}\) is
    \[
		f_{\rm goe}(\mG)
		= \frac{
			\exp(-\frac14\norm{\mG}_{\rm F}^2)
		}{
			(2\pi)^{\frac14 n(n+1)} 2^{\frac n2}
		}.
    \]
\end{fact}

\begin{importedtheorem}[Spectral norm of a GOE, \protect{\cite[Thm.~4.4.3]{vershynin2025high}}]
    \label{impthm:goe-spectral-norm}
    Let \(\mG\sim\mathrm{GOE}(n)\).
    Then, with probability at least \(1-\delta\) we have \(\norm{\mG}_2 \lesssim \sqrt{n} + \sqrt{\log(1/\delta)}\).
\end{importedtheorem}
\begin{importedtheorem}[Trace moments of a GOE, \protect{\cite[Pg.~139]{tao2023topics}}]
    Let \(\mG\sim\mathrm{GOE}(n)\).
    Then,
    \[
        \E[\tr(\mG)] = 0,
        \quad
        \E[\tr(\mG^2)] \lesssim n^2,
        \quad
        \E[\tr(\mG^3)] = 0,\text{ and}
        \quad
        \E[\tr(\mG^4)] \lesssim n^3.
    \]
\end{importedtheorem}

Given the above facts, and the symmetry of the GOE distribution, we can conclude a helpful fact about the trace moments of a GOE conditioned on a bound on its spectral norm.

\begin{lemma}[Spectral norm bounds do not change GOE trace moments]
    \label{lem:goe-moments-conditioned}
    Fix parameters \(n, \delta\).
    Let \(\mG\sim\mathrm{GOE}(n)\).
    Let \(\cA\) be the event that \(\norm{\mG}_2 \lesssim \sqrt{n} + \sqrt{\log(1/\delta)}\)
    Then,
    \[
        \E[\tr(\mG) \mid \cA] = 0,
        \quad
        \E[\tr(\mG^2) \mid \cA] \lesssim n^2,
        \quad
        \E[\tr(\mG^3) \mid \cA] = 0,\text{ and}
        \quad
        \E[\tr(\mG^4) \mid \cA] \lesssim n^3.
    \]
\end{lemma}
\begin{proof}
    For the odd moments, we use that the GOE matrix is rotationally invariant and that the event
    \[
        \cA = \left\{\mG ~:~ \norm{\mG}_2 \lesssim \sqrt n + \sqrt{\log(1/\delta)} \right\}
    \]
    is also rotationally invariant.
    Thereby, conditioning on \cA does not change the expected odd trace moments of \mG.
    For the even moments, we note that \cA happens with probability at least \(1-\delta\) by \cref{impthm:goe-spectral-norm}.
    Further, we have that \(\tr(\mG)^2\) and \(\tr(\mG)^4\) are non-negative.
    So, by \cref{lem:conditioning-expectation}, we know that
    \[
        \E[\tr(\mG^2) \mid \cA]
        \leq 2\E[\tr(\mG^2)]
        \lesssim n^2,
    \]
    with a similar bound for \(\tr(\mG^4)\).
\end{proof}

\section{The proof}
The proof moves in two arcs.
We bound the likelihood ratio between the pdfs of the shifted GOE matrix and the Wishart matrix.
This is a delicate and precise bound, pushed ahead to \cref{sec:ratio-bound}, resulting in the following statement.
\begin{theorem}[Bound on the log-pdf ratio]
    \label{thm:log-ratio-bound}
    Let \mA be a psd matrix.
    Let \(f_{\rm shifted}\) be the pdf associated with the shifted GOE matrix, and let \(f_{\rm wishart}\) be the Wishart pdf.
    Then, for any PSD matrix \(\mA \succeq \frac k2 \mI\), the log pdf ratio \(\alpha(\mA) \defeq \log(f_{\rm shifted}(\mA) / f_{\rm wishart}(\mA))\) satisfies
    \[
        \alpha(\mA)
        \lesssim \frac{n}{k}\tr(\mA-k\mI)
    	- \frac{n}{k^2}\tr((\mA-k\mI)^2)
    	- \frac{1}{k^2}\tr((\mA-k\mI)^3)
    	+ \frac{k}{k^4}\tr((\mA-k\mI)^4)
        + \frac{n^3}{k}.
    \]
\end{theorem}

Before proving \cref{thm:log-ratio-bound}, we show that it suffices to complete the main proof.
\subsection{Proof of \texorpdfstring{\cref{thm:main}}{the main result}}
Let \(\widehat\mG = \sqrt k \mG + k \mI\) be our shifted GOE matrix, so that \(\mG \sim \mathrm{GOE}(n)\).
Let \(\cA\) denote the event that \(\norm{\mG}_2 \leq \frac12\sqrt k\).
By \cref{impthm:goe-spectral-norm}, and since \(k \geq n\), we know that \cA fails with probability at most \(\delta = e^{-k}\).
Our goal is to show that the TV distance between \(\widehat\mG\) and a matrix \(\mW\sim\mathrm{wishart}(n,k)\) is small.
We start by conditioning the \(\widehat\mG\) on \cA and applying Pinsker's inequality.
Let \(\widehat\mG_+\) denote the distribution of \(\widehat\mG\) conditioned on \cA holding.
That is, by \cref{lem:conditioning-tv} and \cref{pinskers-ineq},
\[
    \TV(\widehat\mG, \mW)
    \leq \TV(\widehat\mG_+, \mW) + e^{-k}
    \leq \sqrt{\frac12\KL(\widehat\mG_+ \| \mW)} + e^{-k}.
\]
So, we turn our attention to bounding this KL Divergence.
Remark that for any set \(\cS\subseteq\bbR^{n \times n}\),
\[
    \Pr[\widehat\mG_+ \in \cS]
    = \Pr[\widehat\mG \in \cS \mid \cA]
    = \frac{\Pr[\widehat\mG \in \cS, \cA]}{\Pr[\cA]}
    \leq \frac{\Pr[\widehat\mG \in \cS]}{\Pr[\cA]}
    \leq (1+2e^{-k}) \Pr[\widehat\mG \in\cS]
\]
and therefore that the pdf \(f_+(\mA)\) of \(\widehat\mG_+\) satisfies \(f_+(\mA) \leq (1+2e^{-k})f_{\mathrm{shifted}}(\mA)\).
Thereby, the KL divergence between \(\widehat\mG_+\) and \(\mW\) is
\[
    \KL(\widehat\mG_+ \| \mW)
    = \E\left[\log\left(\frac{f_+(\widehat\mG_+)}{f_{\rm wishart}(\widehat\mG_+)}\right)\right]
    \leq \E\left[\log\left(\frac{f_{\rm shifted}(\widehat\mG_+)}{f_{\rm wishart}(\widehat\mG_+)}\right)\right] + e^{-k},
\]
where we have used that \(\log(1+2e^{-k}) \leq e^{-k}\).
Next, since \cA ensures that \(\norm{\mG}_2 \lesssim \sqrt k\), we know that
\[
    \widehat\mG_+
\succeq \frac k2 \mI.
\]
So, we can apply \cref{thm:log-ratio-bound} to the matrix \(\widehat\mG_+\).
In particular, since \(\widehat\mG_+ - k\mI = k\mG\) conditioned on \cA holding, we have
\[
    \KL(\widehat\mG_+ \| \mW)
    \lesssim \E\left[\frac{n}{k}\tr(\sqrt k \mG)
    	- \frac{n}{k^2}\tr((\sqrt k \mG)^2)
    	- \frac{1}{k^2}\tr((\sqrt k \mG)^3)
    	+ \frac{k}{k^4}\tr((\sqrt k \mG)^4)
         \mid \cA \right] + \frac{n^3}{k} + e^{-k}.
\]
Applying \cref{lem:goe-moments-conditioned} then tells us that
\[
    \KL(\widehat\mG_+ \| \mW)
    \leq \frac{n^3}{k} + e^{-k}.
\]
Returning to our total variation bound, we find that
\[
    \TV(\widehat\mG, \mW)
    \lesssim \sqrt{\frac{n^3}{k}} + e^{-k}
    \lesssim \sqrt{\frac{n^3}{k}}.
\]

\subsection{Proof of \texorpdfstring{\cref{thm:log-ratio-bound}}{the log-ratio bound}}
\label{sec:ratio-bound}
We begin by invoking Stirling's approximation, which says that \(\log\Gamma(x) \leq h(x) + \cO(1/x)\), where
\[
    h(x) = \left(x-\frac12\right)\ln(x) - x + \frac12 \ln(2\pi).
\]
If we substitute in the value \(x=(k+1-i)/2\) where \(i\in\{1,\ldots,n\}\), so that
\begin{align*}
    h\left( \frac{k+1-i}{2} \right)
    = \frac12 \left(
        \log(2\pi)
        - k
        + (i-1)
        + (k-i)\log\left(\frac{k-(i-1)}{2}\right)
    \right).
\end{align*}
The logarithm on the right can be bounded as
\[
    \log\left(\frac{k}{2} \cdot \left(1 - \frac{i-1}{k}\right)\right)
    = \log\left(\frac k2\right) + \log\left(1 - \frac{i-1}{k}\right)
    \leq \log\left(\frac k2\right) - \frac{i-1}{k},
\]
which in turn yields
\[
    (k-i)\log\left(\frac{k-(i-1)}{2}\right)
    \leq (k-i)\log(k) - (k-i)\log(2) - \left(1-\frac{i}{k}\right)(i-1).
\]
We then can bound our use of Stirling's approximation as
\[
    h\left( \frac{k+1-i}{2} \right)
    \leq \frac12 \left(
        \log(2\pi)
        - k
        + (k-i)\log(k)
        - (k-i)\log(2)
        + \cO\left(\frac{i^2}{k}\right)
    \right).
\]
Then, summing over all \(i\in\{1,\ldots,n\}\), we get
\[
    \sum_{i=1}^n
    h\left( \frac{k+1-i}{2} \right)
    \leq \frac12 \left(
        \left(\frac{n(n+3)}{2} - nk\right)\log(2)
        + n\log(\pi)
        - nk
        + \left(nk - \frac{n(n+1)}{2}\right)\log(k)
    \right)
    + \cO\left(\frac{n^3}{k}\right).
\]
We can then use this bound to examine the log of the pdf of a Wishart matrix.
Since \mA is psd, we have
\begin{align*}
    &2\log(f_{\rm wishart}(\mA)) \\
    &=  (k-n-1)\log\det(\mA)
        -\tr(\mA)
        -nk\log(2)
        -\frac{n(n-1)}{2}\log(\pi)
        -\sum_{i=1}^n \log\left(\Gamma\left(\frac{k+1-i}{2}\right)\right)
    \\
    &\geq
        (k-n-1)\log\det(\mA)
        -\tr(\mA)
        -nk\log(2)
        -\frac{n(n-1)}{2}\log(\pi)
        -\sum_{i=1}^n h\left(\frac{k+1-i}{2}\right)
    - \cO\left(\frac{n}{k}\right) \\
    &\geq
        (k-n-1)\log\det(\mA)
        -\tr(\mA)
        -\frac{n(n+3)}{2}\log(2)
        -\frac{n(n+1)}{2}\log(\pi)
        - \left(nk - \frac{n(n+1)}{2}\right)\log(k)
        + nk
- \cO\left(\frac{n^3}{k}\right).
\end{align*}
We now look at the the shifted matrix \(\widehat \mG = \sqrt k \mG + k \mI\) where \(\mG\sim\mathrm{GOE}(n)\).
This matrix has pdf
\[
    f_{\rm shift} (\mA) = \frac{\exp(-\frac1{4k}\norm{\mA-k\mI}_{\rm F}^2)}{(2\pi k)^{\frac14 n(n+1)} 2^{\frac n2}}.
\]
So, we can expand the log of its pdf as
\[
    2\log(f_{\rm shift}(\mA))
    = \frac{-1}{2k}\norm{\mA-k\mI}_{\rm F}^2
    - \frac{n(n+3)}{2}\log(2)
    - \frac{n(n+1)}{2}\log(\pi k).
\]
Then we can observe the log of the ratio between the two pdfs, for a psd matrix \mA, to be
\begin{align*}
    \alpha(\mA)
    &\defeq \log \left(\frac{f_{\rm shift}(\mA)}{f_{\rm wishart}(\mA)}\right) \\
    &\leq \frac12\left\{
        \frac{-1}{2k}\norm{\mA-k\mI}_{\rm F}^2
        -(k-n-1)\log\det(\mA)
        +\tr(\mA)
        +n(k-n-1)\log(k)
        -nk
    \right\} + \cO\left(\frac{n^3}{k}\right).
\end{align*}
Let \(\lambda_i\) denote the \(i^{th}\) eigenvalue of \(\mA\).
Then, we can expand this log ratio in terms of the sum
\[
    \alpha(\mA)
    \leq \frac12
    \sum_{i=1}^n
    \left(
        \frac{-1}{2k}(\lambda_i - k)^2
        -(k-n-1)\log(\lambda_i/k)
        +\lambda_i-k
    \right)
    + \cO\left(\frac{n^3}{k}\right).
\]
Hence, if we introduce a function
\[
    \beta(x) = \frac{-1}{2k}(x - k)^2 - (k-n-1)\log(x/k) + (x-k),
\]
then \(\alpha(\mA) \leq \frac12 \sum_i \beta(\lambda_i(\mA)) + \cO(n^3/k)\).
We now construct the degree-3 Taylor approximation to \(\beta\) centered at \(x=k\).
To do this, we compute the derivatives
\[
\beta(k) = 0
	\qquad
	\beta'(k) = \frac{n+1}{k}
	\qquad
	\beta''(k) = -\frac{n+1}{k^2}
	\qquad
	\beta'''(k) = -\frac{2(k-n-1)}{k^3}
	\qquad
	\beta''''(x) = \frac{6(k-n-1)}{k^4},
\]
which in turn implies that for all \(x\in\bbR\), there exists a value \(\xi\) between \(x\) and \(k\) such that
\[
	\beta(x)
	= \frac{n+1}{k}(x-k)
	- \frac{n+1}{2k^2}(x-k)^2
	- \frac{k-n-1}{3k^3}(x-k)^3
	+ \frac{k-n-1}{4\xi^4}(x-k)^4.
\]
Note that we will evaluate \(\beta(x)\) only at the eigenvalues of \mA, which we are guaranteed to be at least \(k/2\).
Therefore, we know that \(\xi > k/2\) in our case, and we find that
\[
	\beta(\lambda_i)
	\leq \frac{n+1}{k}(\lambda_i-k)
	- \frac{n+1}{2k^2}(\lambda_i-k)^2
	- \frac{k-n-1}{3k^3}(\lambda_i-k)^3
	+ \frac{4(k-n-1)}{k^4}(\lambda_i-k)^4.
\]
Summing over all eigenvalues and returning to our log-ratio function, we find
\[
    \alpha(\mA) \lesssim \frac{n}{k}\tr(\mA-k\mI)
	- \frac{n}{k^2}\tr((\mA-k\mI)^2)
	- \frac{1}{k^2}\tr((\mA-k\mI)^3)
	+ \frac{k}{k^4}\tr((\mA-k\mI)^4)
    + \cO\left(\frac{n^3}{k}\right),
\]
completing the proof.

\bibliographystyle{halpha}
{\footnotesize\newcommand{\etalchar}[1]{$^{#1}$}

 }

\end{document}